\documentclass[final, 1p, times]{elsarticle}

\usepackage{amssymb}
\usepackage{amsthm}
\usepackage{amscd}
\usepackage{amsmath}
\usepackage{amsfonts}
\usepackage{graphicx}
\usepackage{color}
\usepackage{mathrsfs}
\usepackage{titletoc}
\usepackage{titlesec}

\newtheorem{theorem}{Theorem}
\newtheorem{lemma}{Lemma}

\newcommand{\ri}{\right}
\usepackage{geometry}
\usepackage[pagewise,mathlines]{lineno}

\begin{document}		
	\begin{frontmatter}
	\title{A Trudinger-Moser inequality with mean value
zero \\on a compact Riemann surface with boundary}
	\author{Mengjie Zhang}
	\ead{zhangmengjie@ruc.edu.cn}
	\address{School of Mathematics, Renmin University of China, Beijing 100872, P.R.China}		
	\begin{abstract}
In this paper, on a compact Riemann surface $(\Sigma, g)$ with smooth boundary $\partial\Sigma$, we concern a Trudinger-Moser inequality with mean value
zero.
To be exact, let $\lambda_1(\Sigma)$
denotes the first eigenvalue of the Laplace-Beltrami operator with respect to the zero mean value condition and
 $\mathcal{ S }= \left\{ u \in W^{1,2} (\Sigma, g) : \|\nabla_g u\|_2^2 \leq 1\right.$ and $\left.\int _{\Sigma}  u \,dv_g = 0 \right \},$ where $W^{1,2}(\Sigma, g)$ is the usual Sobolev space, $\|\cdot\|_2$ denotes the standard $L^2$-norm
and $\nabla_{g}$ represent the gradient.
By the method of blow-up analysis, we obtain
\begin{eqnarray*}
\sup _{u \in \mathcal{S}} \int_{\Sigma} e^{ 2\pi u^{2}\left(1+\alpha\|u\|_2^{2}\right) }d v_{g}
<+\infty, \ \forall \ 0 \leq\alpha<\lambda_1(\Sigma);
\end{eqnarray*}
when $\alpha \geq\lambda_1(\Sigma)$, the supremum is infinite.
Moreover, we prove the supremum is attained by a function $u_{\alpha} \in C^\infty\left(\overline{\Sigma}\right)\cap \mathcal {S}$ for sufficiently small $\alpha> 0$.
Based on the similar
work in the Euclidean space, which was accomplished by Lu-Yang \cite{Lu-Yang},
we strengthen the result of Yang \cite{Yang2006IJM}.

\end{abstract}
		
		
	\end{frontmatter}

\section{Introduction}
Let $\Omega\subseteq \mathbb{R}^2$ be a smooth bounded domain and $W_0^{1,2}(\Omega)$ be the completion of $C_0^{\infty}(\Omega)$ under the Sobolev norm
$\|\nabla_{\mathbb{R}^2} u\|_2^2= \int_{\Omega}{|\nabla_{\mathbb{R}^2} u|^2}dx,$
where $\nabla_{\mathbb{R}^2}$ is the gradient operator on ${\mathbb{R}^2}$ and $\|\cdot\|_2$ denotes the standard $L^2$-norm.
The classical Trudinger-Moser inequality \cite{Yudovich, Pohozaev, Peetre, Trudinger1967, Moser1970}, as the limit case of the Sobolev embedding, says
    \begin{eqnarray}\label{Trudinger-Moser}
    \sup_{u\in W_0^{1, 2}(\Omega), \, \|\nabla_{\mathbb{R}^2} u\|_2\leq 1}
    \int_\Omega e^{\ \beta u^2}dx<+\infty, \ \forall \ \beta\leq 4\pi.
    \end{eqnarray}
Moreover, $4\pi$ is called the best
constant for this inequality in the sense that when $\beta> 4\pi$, all integrals in (\ref{Trudinger-Moser}) are still finite, but the supremum is infinite.
It is interesting to know whether or not the supremum in (\ref{Trudinger-Moser}) can be attained.
For this topic, we refer the reader to Carleson-Chang \cite{C-C},
Flucher \cite{Flucher}, Lin \cite{Lin}, Struwe \cite{Struwe}, Adimurthi-Struwe \cite{A-Struwe}, Li \cite{Li-Sci}, Yang \cite{Yang-JFA-06}, 
Zhu \cite{ZhuJY}, Tintarev \cite{Tintarev} and the references therein.

There are many extensions of (\ref{Trudinger-Moser}).
Adimurthi-Druet \cite{AD} generalized (\ref{Trudinger-Moser}) to the following form:
    \begin{equation}\label{T-AM}
  \sup _ { u \in W _ { 0 } ^ { 1,2 } ( \Omega ) , \| \nabla_{\mathbb{R}^2} u \| _ { 2 } \leq 1 } \int _ { \Omega } e ^ { 4 \pi u ^ { 2 } \left( 1 + \alpha \| u \| _ { 2 } ^ { 2 } \right) } dx <+\infty, \ \forall\ 0 \leq \alpha < \lambda _ { 1 } ( \Omega ),
\end{equation}
where $\lambda _ { 1 } ( \Omega )$ is the first eigenvalue of the Laplacian with Dirichlet boundary condition in $\Omega .$
This inequality is sharp in the sense that if $\alpha \geq \lambda _ { 1 } ( \Omega )$, all integrals in (\ref{T-AM}) are still finite, but the supremum is infinite. Obviously, (\ref{T-AM}) is reduced to (\ref{Trudinger-Moser}) when $\alpha=0$.
Various extensions of the inequality (\ref{T-AM}) were obtained by Yang \cite{Yang-JFA-06,Yang-JDE-15}, Tintarev \cite{Tintarev} and Zhu \cite{ZhuJY} respectively.
It was extended by Lu-Yang \cite{Lu-Yang} to a version, namely
\begin{equation}\label{T-LY}
\sup _ {u \in W^{1,2}(\Omega),
\int_{\Omega} u dx=0, \| \nabla_{\mathbb{R}^2} u \| _ { 2 } \leq 1 } \int _ { \Omega } e ^ { 2 \pi u ^ 2 \left( 1 + \alpha \| u \| _ 2 ^ 2 \right)  } d x < + \infty, \ \forall\ 0 \leq \alpha < \overline{\lambda} _ 1 ( \Omega ),
\end{equation}
where 
$ \overline{\lambda} _ 1 ( \Omega )$ denotes the first nonzero Neumann eigenvalue of the Laplacian operator.
This inequality is sharp in the sense that all integrals in (\ref{T-LY}) are still finite when $ \alpha \geq  \overline{\lambda} _ 1 ( \Omega)$, but the supremum is infinite.
Moreover, for sufficiently small $\alpha> 0$, the supremum is attained.

Trudinger-Moser inequalities were introduced on Riemannian manifolds by Aubin \cite{A}, Cherrier \cite{C} and Fontana \cite{Fontana}. In particular, let $( \Sigma , g )$ be a 2-dimensional compact Riemann surface, $W^ { 1 , 2 } ( \Sigma, g) $ the completion of $C ^ { \infty } ( \Sigma)$ under the norm
$\|u \| ^2_ { W^{1,2}( \Sigma, g) } =  \int _ { \Sigma }( u^2+| \nabla_g u | ^2)\, dv_g$, where $\nabla_g$ stands for the gradient operator on $(\Sigma, g)$.
When $( \Sigma , g )$ is closed Riemann surface,
there holds
\begin{equation}\label{T-MM}
  \sup _ {u \in W ^ { 1,2 } ( \Sigma, g), \int _\Sigma  u dv_g = 0,  \| \nabla_g u \| _ { 2 } \leq 1  }\int _ { \Sigma } e ^ {\ \beta u^ 2} dv_g<+\infty, \ \forall \ \beta\leq 4\pi.
    \end{equation}
Moreover, $4\pi$ is called the best
constant for this inequality in the sense that when $\beta> 4\pi$, all integrals in (\ref{T-MM}) are still finite, but the supremum is infinite.
Based on the works of Ding-Jost-Li-Wang \cite{DJLW} and Adimurthi-Struwe \cite{A-Struwe}, Li \cite{Li-JPDE,Li-Sci} proved the existence of extremals for the supremum in (\ref{T-MM}).
When $(\Sigma, g)$ is a compact Riemann surface with smooth boundary $\partial\Sigma$,
Yang \cite{Yang2006IJM} obtained the same inequality as (\ref{T-MM}), namely
\begin{equation}\label{2006Y}
  \sup _ {u \in W ^ { 1,2 } ( \Sigma, g), \int _\Sigma  u dv_g = 0,  \| \nabla_g u \| _ { 2 } \leq 1  }\int _ { \Sigma } e ^ {\ \beta u^ 2} dv_g<+\infty, \ \forall \ \beta\leq 2\pi.
    \end{equation}
This inequality is sharp in the sense that if $\beta> 2\pi$, all integrals in (\ref{2006Y}) are still finite, but the supremum is infinite. Furthermore, the supremum in (\ref{2006Y}) can be attained.
\\

In view of the inequality (\ref{T-LY}) in the Euclidean space, we strengthen (\ref{2006Y}) on $(\Sigma, g)$ with smooth boundary $\partial\Sigma$. Precisely we have the following:
\begin{theorem}\label{T1}
  Let $(\Sigma, g)$ be a compact Riemann surface with smooth boundary $\partial\Sigma$
  and
  \begin{eqnarray}\label{la}
\lambda_1(\Sigma)= \inf _ { u \in W^{1,2} ( \Sigma, g) , \int _ { \Sigma } u dv_g = 0 , u \not\equiv 0 } \frac { \|\nabla_g u\|^2_2} {\|u\|_2^2}
\end{eqnarray}
be the first eigenvalue of the Laplace-Beltrami operator $\Delta _ { g}$ with respect to the zero mean value condition.
Denote a function space
\begin{eqnarray*}\label{S}
\mathcal { S } = \left\{ u \in W ^ { 1,2 } ( \Sigma, g) : \int _\Sigma  {u}\ dv_g = 0, \ \| \nabla_g u \| _ { 2 } \leq 1  \right\}
\end{eqnarray*}
and
\begin{eqnarray*}
F_{\alpha}^{\beta}(u)= \int_{\Sigma} e^{ \beta u^{2}\left(1+\alpha\|u\|_2^{2}\right) }d v_g.
\end{eqnarray*}
Then there hold\\
(i) for any $ \alpha \geq\lambda_1(\Sigma),$ $\sup _{u \in \mathcal{S}}
F_{\alpha}^{2\pi}(u)=+\infty$;\\
(ii) for any $0 \leq \alpha<\lambda_1(\Sigma)$, $\sup _{u \in \mathcal{S}} F_{\alpha}^{2\pi}(u)<+\infty$;\\
(iii) for sufficiently small $\alpha>0$, $\sup _{u \in \mathcal{S}} F_{\alpha}^{2\pi}(u)$ can be attained by some function $u_{\alpha} \in C^\infty\left(\overline{\Sigma}\right)\cap \mathcal {S}$.
\end{theorem}

For the proof,
we employ the method of blow-up analysis, which was originally used by Carleson-Chang\cite{C-C},
Ding-Jost-Li-Wang \cite{DJLW}, Adimurthi-Struwe \cite{A-Struwe}, Li \cite{Li-JPDE}, 
and Yang \cite{Yang2007, Yang-JDE-15}.
For related works, we refer the reader to
Adimurthi-Druet \cite{AD},
do \'O-de Souza \cite{do-de2014,do-de2016}, Nguyen \cite{N2017,N2018}, Li-Yang \cite{L-Y}, Zhu \cite{Zhu}, Fang-Zhang \cite{F-Z}, Yang-Zhu \cite{Yang-Zhu2018,Yang-Zhu2019} and Csat\'o-Nguyen-Roy \cite{C-N-R}.
We should point out that the blow-up occurs on the boundary $\partial\Sigma$ in our case.
The key ingredient in the proof of our theorem is the isothermal coordinate system on $\partial\Sigma$.
Though such coordinates have been used by many authors  (see for example Li-Liu \cite{Li-Liu} and Yang \cite{Yang2006IJM, Yang2006PJM, Yang-2007}), the proof of its existence around has just been provided by Yang-Zhou \cite{Yang-Zhou} via Riemann mapping theorems involving the boundary.

The remaining part of this paper will be organized as follows: In Section 2, we prove (Theorem \ref{T1}, ($i$)) by constructing test functions; in Section 3, we prove (Theorem \ref{T1}, ($ii$)) by using blow-up analysis; in Section 4, we construct a sequence of functions to show (Theorem \ref{T1}, ($iii$)) holds.
Hereafter we do not distinguish the sequence and the subsequence; moreover, we often denote various constants by the same $C$.

\section{The case of $\alpha\geq\lambda_1(\Sigma)$}

In this section, we select test functions to prove Theorem \ref{T1} ($i$).
Let $\lambda_1(\Sigma)$ be defined by (\ref{la}) and $\alpha \geq \lambda_1(\Sigma)$.
From a direct method of variation,
one obtains that
there exists some function $u_0 \in \mathcal { S } $, such that
\begin{eqnarray}\label{s3}
\lambda_1(\Sigma) =\left\|\nabla_g u_0\right\|_{2}^{2}.
\end{eqnarray}
By a direct calculation, we derive that $u_0$ satisfies the Euler-Lagrange equation
\begin{eqnarray}\label{e-ua}
	\left\{
\begin{aligned}
&\Delta_g u _ 0 =\lambda_1(\Sigma)\,u _0 \,\,\, \mathrm{in}\,\,\, \Sigma ,\\
&{ \frac { \partial u _0 } { \partial \mathbf{n} } =0 \,\,\, \mathrm{ on } \,\,\, \partial \Sigma },\\
&\int _ { \Sigma } u_0 \,dv_g = 0,\ \int _{\Sigma}  {u_0^2} \,dv_g=1,
 \end{aligned} \right.
\end{eqnarray}
where $\mathbf{n}$ denotes the outward unit normal vector on $\partial\Sigma$. Applying elliptic estimates to (\ref{e-ua}), we obtain $u_0 \in \mathcal { S } \cap C^{\infty}\left(\overline{\Sigma}\right)$.
Consequently, there exist a point $ x_0\in \partial\Sigma$ with $u _ { 0 } ( x_0 ) > 0 $ and
a neighborhood $U$ of $x_0$ with $u_ { 0 }(x) \geq u_ { 0 } ( x_0 ) / 2$ in $U$ .
Let $\delta = {\left( t _\epsilon  \sqrt { -\log  \epsilon  } \right)}^{-1}$, where $t _ { \epsilon } > 0$ such that $-t _ { \epsilon } ^ { 2 } \log  { \epsilon } \rightarrow + \infty$ and $t _ { \epsilon } ^ { 2 } \sqrt { -\log { \epsilon } } \rightarrow 0$ as $\epsilon\rightarrow0$.
We take an isothermal coordinate system $\left(\phi^{-1}\left(\mathbb{B}_\delta^+\right),\phi\right)$ such that $\phi(x_0) = 0$ and $\phi^{-1}\left(\mathbb{B}_\delta^+\right)\subset U $.
In such coordinates, the metric $g$ has the representation $g = e^{2f}\left (dx_1^2 +dx_2^2 \right)$ and $f$ is a smooth function with $f(0)  = 0$.

On $\overline{\mathbb{B}_{\delta}^{+}},$ we define a sequence of functions
\begin{eqnarray*}
\tilde{u}_{\epsilon}(x)=\left\{
\linespread{1.5}\selectfont
\begin{aligned}
&\sqrt{\frac{-\log \epsilon}{2 \pi} }, \ \ \ \ \ \ \ \ \ \ \  \  |x| \leq \delta \sqrt{\epsilon}, \\
&\sqrt{\frac{-2}{\pi \log \epsilon}} \log \frac{\delta}{|x|},\  \delta \sqrt{\epsilon}<|x| \leq \delta.
\end{aligned}\right.
\end{eqnarray*}
And we set
\begin{eqnarray}\label{s4}
u_{\epsilon}=\left\{\begin{array}{ll}
\tilde{u}_{\epsilon} \circ \phi & \text { in }\ \ \phi ^{-1}\left(\overline{\mathbb{B}_{\delta}^{+}}\right), \\
s_{\epsilon}\, \varphi & \text { in } \ \ \Sigma \backslash \phi ^{-1}\left(\overline{\mathbb{B}_{\delta}^{+}}\right),
\end{array}\right.
\end{eqnarray}
where $\varphi \in C_{0}^{\infty}\left(\Sigma \backslash \phi^{-1}(\mathbb{B}_\delta)\right)$ and $s_{\epsilon}$ is a real number such that $\int_{\Sigma} u_{\epsilon} \,dv_g=0$.
Set $v_{\epsilon}=u_{\epsilon}+t_{\epsilon} u_{0}$. According to (\ref{s3})-(\ref{s4}), we have
\begin{eqnarray}\label{s5}
\left\|v_{\epsilon}\right\|_{2}^{2}=\left\|u_{\epsilon}\right\|_{2}^{2}+t_{\epsilon}^{2}\left\|u_{0}\right\|_{2}^{2}+2 t_{\epsilon} \int_{\Sigma} u_{\epsilon} u_{0} \,dv_g
= t_{\epsilon}^{2}+2 t_{\epsilon} \int_{\Sigma} u_{\epsilon} u_{0} \,dv_g+O\left(\frac{-1}{ \log\epsilon}\right) \end{eqnarray}
and
\begin{eqnarray}\label{s6}
\left\|\nabla_g v_{\epsilon}\right\|_{2}^{2}
= 1+\lambda_1(\Sigma) t_{\epsilon}^{2}+2 \lambda_1(\Sigma) t_{\epsilon} \int_{\Sigma} u_{\epsilon} u_{0} \,dv_g+O\left(\frac{-1}{ \log\epsilon}\right).
\end{eqnarray}
Take $v_{\epsilon}^{*} =v_{\epsilon} /\lVert \nabla_gv_{\epsilon} \rVert _{2}^{2}\in\mathcal{S}$.
From $\alpha \geq \lambda_1(\Sigma)$ and (\ref{s4})-(\ref{s6}), we have that on $\phi^{-1}\left(\mathbb{B}_{\delta \sqrt{\epsilon}}^{+}\right)$
\begin{eqnarray*}
&\ &{2 \pi v_{\epsilon}^{*2} \left(1+\alpha\left\|v_{\epsilon}^*\right\|_{2}^{2}
\right)}\\
&=&
2 \pi {v_{\epsilon}^{2}}\frac{1}{\left\|\nabla_g v_{\epsilon}\right\|_{2}^{2}} \left(1+\alpha\frac{\left\|v_{\epsilon}\right\|_{2}^{2}}
{\left\|\nabla_g v_{\epsilon}\right\|_{2}^{2}}\right)\\
 & \geq &2 \pi\left(t_{\epsilon}^{2} u_{0}^{2}-\frac{\log \epsilon}{2 \pi} +2 t_{\epsilon}\sqrt{\frac{-\log \epsilon}{2 \pi} } u_{0}\right)
 \left(1+\left(\alpha-\lambda_1(\Sigma)\right)\left(t_{\epsilon}^{2}+2 t_{\epsilon} \int_\Sigma u_{\epsilon} u_{0}\,dv_g\right)+o\left(\frac{t_{\epsilon}}{\sqrt{-\log \epsilon}}\right)\right)
 \\
 & \geq &\left(2 \pi t_{\epsilon}^{2} u_{0}^{2}-\log \epsilon +4 \pi t_{\epsilon}\sqrt{\frac{-\log \epsilon}{2 \pi} } u_{0}\right)\left(1+o\left(\frac{t_{\epsilon}}{\sqrt{-\log \epsilon}}\right)\right) \\
& \geq& -\log {\epsilon}+t_{\epsilon}\sqrt{-\log \epsilon}\left(\sqrt{8 \pi} u_{0}+o(1)\right).
\end{eqnarray*}
Hence there holds
\begin{eqnarray*}
\int_{\Sigma} e^{2 \pi v_{\epsilon}^{*2} \left(1+\alpha\left\|v_{\epsilon}^*\right\|_{2}^{2}
\right)} \,dv_g & \geq& \int_{\phi^{-1}\left(\mathbb{B}_{\delta \sqrt{\epsilon}}^{+}\right)} \frac{1}{\epsilon} e^{t_{\epsilon}\sqrt{-\log \epsilon}\left(\sqrt{8 \pi} u_{0}+o(1)\right)} \,dv_g \\
& \geq &C(\delta) e^{t_{\epsilon} \sqrt{-\log {\epsilon}}\left(\sqrt{2 \pi} u_{0}(x_0)+o(1)\right)}
\end{eqnarray*}
for some positive constant $C(\delta) .$ In view of $u_{0}(x_0)>0,$ we get $ \sup _ { u \in \mathcal {S} } F_\alpha^{2\pi}(u) \ge \lim_{\epsilon \rightarrow 0}F_\alpha^{2\pi}\left(v_{\epsilon}^{*}\right)=+\infty$.
This completes the proof of Theorem \ref{T1} ($i$).

\section{The case of\, $0\leq\alpha<\lambda_1(\Sigma)$}
In this section, we will prove Theorem \ref{T1} ($ii$) in three steps:
firstly, we consider the existence of maximizers for subcritical functionals and give the corresponding Euler-Lagrange equation;
secondly, we deal with the asymptotic behavior of the maximizers through blow-up analysis;
finally, we deduce an upper bound of the supremum $\sup_{ u \in \mathcal { S } }F_{\alpha}^{2\pi}(u)$ under the assumption that blow-up occurs.

\subsection*{\textbf{Step 1.} Existence of maximizers for subcritical functionals}

Using the similar proof of (\cite{Lu-Yang}, Step 1), we have the following
\begin{lemma}\label{L1}
For any $\epsilon>0$, there exists some function $ u_{\epsilon} \in \mathcal {S}\cap C^{\infty}\left(\overline{ \Sigma}\right)$ with $\|\nabla_g u_{\epsilon}\|_2^2=1$, such that
\begin{eqnarray*}
\sup_{u \in\mathcal { S}}F_{\alpha}^{2\pi-\epsilon}(u)
=F_{\alpha}^{2\pi-\epsilon}(u_\epsilon).
\end{eqnarray*}
Moreover, $u_{\epsilon}$ satisfies the Euler-Lagrange equation
\begin{eqnarray}\label{e-u}
	\left\{ \begin{aligned}
	& \frac{\partial u_{\epsilon}}{\partial \mathbf{n}}=0\,\,\mathrm{on\,\,}{\partial \Sigma},\\
	&\Delta_g u_{\epsilon}=\frac{\beta _{\epsilon}}{\lambda _{\epsilon}}u_{\epsilon}e^{\alpha _{\epsilon}u_{\epsilon}^{2}}+\gamma _{\epsilon}u_{\epsilon}-\frac{\mu _{\epsilon}}{\lambda _{\epsilon}}\,\,\mathrm{in\,\,}\Sigma,
\\
	&\alpha _{\epsilon}=\left( 2\pi -\epsilon \right) \left( 1+\alpha\|u_{\epsilon}\|_2^2 \right),\ \
	\beta _{\epsilon}=\frac{1+\alpha \|u_{\epsilon}\|_2^2}{1+2\alpha \|u_{\epsilon}\|_2^2},\\
&\gamma _{\epsilon}=\frac{\alpha}{1+2\alpha \|u_{\epsilon}\|_2^2},\ \
	\lambda _{\epsilon}=\int_{\Sigma}{u}_{\epsilon}^{2}e^{\alpha _{\epsilon}u_{\epsilon}^{2}}\,dv_g,\ \ \mu _{\epsilon}=\frac{\beta _{\epsilon}}{{\rm Area}(\Sigma )}\int_{\Sigma}{u}_{\epsilon}e^{\alpha _{\epsilon}u_{\epsilon}^{2}}\,dv_g.
\end{aligned} \right.
	\end{eqnarray}
\end{lemma}

It follows from Lebesgue's dominated convergence theorem that
    \begin{eqnarray}\label{3}
    {\lim_{\epsilon \rightarrow 0}}
    F_{\alpha}^{2\pi-\epsilon}(u_\epsilon)=\sup_{u \in\mathcal { S}}F_{\alpha}^{2\pi}(u).
    \end{eqnarray}
Seeing the fact of $1+t e^{t}\geq e^{t} $ for any $t \geq 0$, we get
\begin{eqnarray*}
\lambda _\epsilon=\int_{\Sigma}{u}_\epsilon^{2}e^{\alpha _\epsilon u_\epsilon^{2}}\,dv_g
\geq \frac{1}{\alpha_\epsilon}\int_{\Sigma}\left(e^{\alpha_\epsilon u_\epsilon^{2}}-1\right) d v_g,
\end{eqnarray*}
which together with (\ref{3}) leads to
\begin{eqnarray}\label{L2}
  \liminf_{\epsilon \rightarrow 0} \ \lambda _{\epsilon}>0.
\end{eqnarray}
In view of (\ref{e-u}), (\ref{L2}) and $\beta _ { \epsilon } \leq 1$, we obtain
\begin{eqnarray}
\left| \frac{\mu _{\epsilon}}{\lambda _{\epsilon}} \right|
\nonumber &\leq&  \frac{1}{\lambda _{\epsilon}{\rm Area}(\Sigma )}\left( \int_{\left\{ u\in \Sigma :\left| u_{\epsilon} \right|\ge 1 \right\}}{|u_{\epsilon}}|e^{\alpha _{\epsilon}u_{\epsilon}^{2}}\,dv_g+\int_{\left\{ u\in \Sigma :\left| u_{\epsilon} \right|<1 \right\}}{|u_{\epsilon}}|e^{\alpha _{\epsilon}u_{\epsilon}^{2}}\,dv_g \right)  \\
\nonumber &\leq& \frac{1}{\lambda _{\epsilon}{\rm Area}(\Sigma )}\left( \int_{\left\{ u\in \Sigma :\left| u_{\epsilon} \right|\ge 1 \right\}}{u_{\epsilon}^{2}}e^{\alpha _{\epsilon}u_{\epsilon}^{2}}\,dv_g+\int_{\left\{ u\in \Sigma :\left| u_{\epsilon} \right|<1 \right\}}{e^{\alpha _{\epsilon}u_{\epsilon}^{2}}\,dv_g} \right) \\
\nonumber &\leq& \frac{1}{{\rm Area}(\Sigma )}+\frac{e^{\alpha _{\epsilon}}}{\lambda _{\epsilon}}\\
\label{4} &\leq& C.
\end{eqnarray}

\subsection*{\textbf{Step 2.} Blow-up analysis}

Since $u_{\epsilon}$ is bounded in $ W^{1,2}( \Sigma , g)$,
there exists some function $u_0\in W^{1,2}( \Sigma , g)$ such that
\begin{eqnarray}\label{s2}
	\left\{ \begin{aligned}
	&u_{\epsilon}\rightharpoonup u_0\, \, \mathrm{weakly}\, \, \mathrm{in}\, \, W^{1, 2}\left(\Sigma,g \right),\\
	&u_{\epsilon}\rightarrow u_0\, \, \mathrm{strongly}\, \, \mathrm{in}\, \, L^p\left( \Sigma,g\right) , \forall p>1,\\
	&u_{\epsilon}\rightarrow u_0\, \, \mathrm{a. e.} \, \, \mathrm{in}\, \, \Sigma.
	\end{aligned} \right.
	\end{eqnarray}
Then we have $\int _ {  \Sigma } u _ { 0 } dv _ { g } = 0$ and $\|\nabla_gu_0\|_2^2\leq1$.

We set $c_{\epsilon}=|u_{\epsilon}( x_\epsilon )| =\max_{\overline{\Sigma}}|u_{\epsilon}|$.
We first assume that $c_{\epsilon}$ is bounded, which together with elliptic estimates completes the proof of Theorem \ref{T1} ($ii$).
Without loss of generality, we assume
\begin{eqnarray}\label{c}
c_{\epsilon}=u_{\epsilon}( x_\epsilon )\rightarrow+ \infty
\end{eqnarray}
and $x_\epsilon\rightarrow x_0$ as $\epsilon\rightarrow0.$
 Applying maximum principle to (\ref{e-u}), we have $x_0 \in \partial \Sigma$.

Following (\cite{Yang-Zhou}, Lemma 4),
we can take an isothermal coordinate system $(U,\phi)$ near $x_0$, such that $\phi(x_0) = 0$, $\phi (U)=\mathbb { B }_r^+$ and $\phi (U\cap \partial\Sigma)= \partial\mathbb { R } ^ { 2  }_+\cap\mathbb { B }_r$ for some fixed $r>0$, where $\mathbb { B }_r^+=\{(x_1,\ x_2)\in \mathbb { R } ^ 2: x_1^2+x_2^2\leq r^2,\ x_2>0 \}$ and $\mathbb{R}^2_+=\left\{x=(x_1,x_2)\in \mathbb { R } ^ 2:x_2>0\right\}$. In such coordinates, the metric $g$ has the representation $g = e^{2f} \left(dx_1^2 +dx_2^2 \right)$ and $f$ is a smooth function with $f (0) = 0$.
Denote $\tilde{x}_\epsilon=\phi(x_\epsilon)$ and $\tilde{u}_\epsilon=u_\epsilon\circ\phi^{-1}$.
To proceed, we observe an energy concentration phenomenon of $u_\epsilon$.
\begin{lemma}\label{L3}
There hold $u_0=0$ and $ |\nabla_g u_\epsilon|^2\,dv_g\rightharpoonup \delta _{x_0} $ in sense of measure, where $\delta _{x_0}$ stands for the Dirac measure centered at $x_0$.
\end{lemma}

\begin{proof}
We first prove $u_0 \equiv 0$. Suppose not, we can see that $0< \lVert \nabla_g u_0 \rVert _{2}^{2}\leq1$. Letting $\eta=\lVert \nabla_g u_0 \rVert _{2}^{2}$,
one has $\lVert \nabla_g \left( u_{\epsilon}-u_0 \right) \rVert _{2}^{2}\rightarrow 1-\eta<1$ and $1 + \alpha \left\| u _ { \epsilon } \right\| _ 2 ^ { 2 } \rightarrow 1 + \alpha \left\| u _ { 0 } \right\| _ 2 ^ { 2 } \leq 1+\eta$ as $\epsilon\rightarrow0$.
For sufficiently small $\epsilon$, we obtain
\begin{eqnarray*}
\left(1 + \alpha \left\| u _ { \epsilon } \right\| _ 2 ^ { 2 } \right)\lVert \nabla_g \left( u_{\epsilon}-u_0 \right) \rVert _{2}^{2}\leq  \frac{2-\eta^2}{2}<1.
\end{eqnarray*}
From the H\"older inequality, there holds
\begin{eqnarray*}
\int_{\Sigma}{e^{q\alpha _{\epsilon}u_{\epsilon}^{2}}}\,dv_g&\leq& \int_{\Sigma}{e^{q\left( 1+\frac{1}{\delta} \right)\alpha _{\epsilon} u_{0}^{2}+q\left( 1+\delta \right)\alpha _{\epsilon} \left( u_{\epsilon}-u_0 \right) ^2}}\,dv_g\\
	&\leq& \left( \int_{\Sigma}{e^{rq\left( 1+\frac{1}{\delta} \right)\alpha _{\epsilon} u_{0}^{2}}}\,dv_g \right) ^{\frac{1}{r}}\left( \int_{\Sigma}{e^{sq \left( 1+\delta \right)\left( 2\pi -\epsilon \right) \left( 1+\alpha \lVert u_{\epsilon}\lVert _2^{2} \right) \lVert \nabla_g \left( u_{\epsilon}-u_0 \right) \rVert _{2}^{2}\frac{\left( u_{\epsilon}-u_0 \right) ^2}{\lVert \nabla_g \left( u_{\epsilon}-u_0 \right) \rVert _{2}^{2}}}}\,dv_g \right) ^{\frac{1}{s}}\\
	&\leq& C\left( \int_{\Sigma}{e^{sq \left( 1+\delta \right) \left( 2\pi -\epsilon \right)\frac{2-\eta^2}{2}  \frac{\left( u_{\epsilon}-u_0 \right) ^2}{\lVert \nabla_g \left( u_{\epsilon}-u_0 \right) \rVert _{2}^{2}}}}\,dv_g \right) ^{\frac{1}{s}}
\end{eqnarray*}
for  sufficiently small $\delta$, some $r,\ s,\ q>1$ satisfying ${1}/{r}+{1}/{s}=1$ and $sq\left( 1+\delta \right) \left( 2-\eta^2 \right)/2 <1 $. In view of the Trudinger-Moser inequality (\ref{2006Y}), we get $e^{\alpha _{\epsilon}u_{\epsilon}^{2}}$ is bounded in $L^q\left( \Sigma, g \right) $. Hence $\Delta_g u_{\epsilon}$ is bounded in some $L^q\left( \Sigma,g \right)$ from (\ref{e-u}) and (\ref{4}). Applying the elliptic estimate to (\ref{e-u}), one gets $u_{\epsilon}$ is uniformly bounded, which contradicts our assumption $c_{\epsilon}\rightarrow +\infty$. That is to say $u_0 \equiv 0$.

Next we prove $|\nabla_g u_{\epsilon}|^2\,dv_g\rightharpoonup \delta_{x_0}$ in sense of measure. Suppose not. There exists some $r>0$ such that \begin{eqnarray*}\lim_{\epsilon \rightarrow 0}\int_{B_{r}(x_0)}{\left| \nabla_g u_{\epsilon} \right|}^2\,dv_g:=\eta <1,\end{eqnarray*}
where ${B_{r}(x_0)}$ is a geodesic ball centered at $x_0$ with radius $r$.
For sufficiently small $\epsilon$, we can see that $\int_{B_{r}(x_0)}{| \nabla _gu_{\epsilon} |}^2\,dv_g\leq (\eta +1)/{2}<1.$
Then we choose a cut-off function $\rho$ in $C_{0}^{1}\left(\phi\left(B_{r_0}(x_0)\right) \right)$, which is equal to $1$ in $\overline{\phi\left(B_{r_0/2}(x_0)\right) }$  such that
\begin{eqnarray*}\int_{\phi \left( B_{r}\left( x_0 \right)\right) }|\nabla_g \left( \rho \tilde{u}_\epsilon \right) |^2dx
\leq \frac{\eta +3}{4}<1.\end{eqnarray*}
Hence we obtain
\begin{eqnarray*}
\int_{ B_{r/2}\left( x_0 \right)  }{e^{\alpha _{\epsilon}qu_{\epsilon}^{2}}}\,dv_g&=&\int_{\phi \left( B_{r/2}\left( x_0 \right) \right)}{e^{\alpha _{\epsilon}q\tilde{u}_\epsilon^{2}}}e^{2f}dx
\\
&\leq& C\int_{\phi \left( B_{r}\left( x_0 \right) \right) }{e^{\alpha _{\epsilon}q\left( \rho \tilde{u}_\epsilon \right) ^2}}dx
\\
&\leq& C\int_{\phi \left( B_{r}\left( x_0 \right)\right) }{e^{\alpha _{\epsilon}q  \frac{\eta +3}{4}  \frac{\left( \rho \tilde{u}_\epsilon \right) ^2}{\int_{\phi \left( B_{r}\left( x_0 \right)\right) }|\nabla_g \left( \rho \tilde{u}_\epsilon \right) |^2dx}}}dx.
\end{eqnarray*}
From the Trudinger-Moser inequality (\ref{2006Y}),  we get $e^{\alpha _{\epsilon}u_{\epsilon}^{2}}$ is bounded in $L^q\left( B_{r/2}(x_0), g\right) $ for any $q>1$ satisfying $q(\eta +3)/4\leq1$. Applying the elliptic estimate to (\ref{e-u}), one gets $u_{\epsilon}$ is uniformly bounded in $B_{r/4}\left( x_0 \right) $. This contradicts (\ref{c}) and ends the proof of the lemma.
	\end{proof}

\begin{lemma}\label{L4}
Let
  \begin{eqnarray}\label{r}
  r _ { \epsilon }  = \sqrt{\frac { \lambda _ { \epsilon } } { \beta _ { \epsilon } c _ { \epsilon } ^ { 2 }  e ^ { \alpha _ { \epsilon } c _ { \epsilon } ^ { 2 } }}}.
  \end{eqnarray}
Then there hold $\lim_{\epsilon\rightarrow0}r_\epsilon= 0$ and $\lim_{\epsilon\rightarrow0}r_\epsilon^2c_\epsilon^k= 0$, where $k$ is a positive integer.
\end{lemma}

\begin{proof}
It is easy to know that $\lim_{\epsilon\rightarrow0}r_\epsilon= 0$ from (\ref{e-u}).
Using the H\"older inequality and (\ref{e-u}), we have
	\begin{eqnarray*}
r_{\epsilon}^2 c_{\epsilon}^k&=&\frac{\lambda _{\epsilon}}{\beta _{\epsilon}c_{\epsilon}^{2-k}e^{\alpha _{\epsilon}c_{\epsilon}^{2}}}
\leq \frac{\int_{\Sigma}{u_{\epsilon}^{k}e^{\left( 1-\delta \right) \alpha _{\epsilon}u_{\epsilon}^{2}}\,dv_g}}{\beta _{\epsilon}e^{\left( 1-\delta \right) \alpha _{\epsilon}c_{\epsilon}^{2}}}
\\
 &\leq& \frac{1}{\beta _{\epsilon}e^{\left( 1-\delta \right) \alpha _{\epsilon}c_{\epsilon}^{2}}}\left( \int_{\Sigma}{u_{\epsilon}^{kr}\,dv_g} \right) ^{\frac{1}{r}}\left( \int_{\Sigma}{e^{\left( 1-\delta \right) s\alpha _{\epsilon}u_{\epsilon}^{2}}\,dv_g} \right) ^{\frac{1}{s}}\\
 &\leq &\frac{1}{\beta _{\epsilon}}\left( \int_{\Sigma}{u_{\epsilon}^{kr}\,dv_g} \right) ^{\frac{1}{r}}\left({\rm Area}(\Sigma ) \right) ^{\frac{1}{s}}
	\end{eqnarray*}
for any $0<\delta <1$ and $1/r+1/s=1$.
Then the lemma follows from Lemma \ref{L3} immediately.
\end{proof}

Define
\begin{eqnarray*}
\tilde{u}_{\epsilon}(x)=\left\{\begin{aligned}
&u_{\epsilon} \circ \phi^{-1}\left(x_{1}, x_{2}\right), &  x_{2} \geq 0, \\
&u_{\epsilon} \circ \phi^{-1}\left(x_{1},-x_{2}\right), &x_{2}<0,
\end{aligned}\right.
\end{eqnarray*}
and
\begin{eqnarray*}
\tilde{f}(x)=\left\{\begin{aligned}
&f \left(x_{1}, x_{2}\right), &  x_{2} \geq 0, \\
&f \left(x_{1},-x_{2}\right), &x_{2}<0,
\end{aligned}\right.
\end{eqnarray*}
on $\mathbb{B}_r$.
Let $U_\epsilon=\left\{x\in\mathbb{R}^2:\ \tilde{x}_\epsilon +r_{\epsilon}x\in\mathbb { B }_r\right\}$. Then one has $U_ { \epsilon }\rightarrow \mathbb{R}^2\,\mathrm{as}\, \epsilon \rightarrow 0$ from Lemma \ref{L4}.
Define two blowing up functions on $ U _ { \epsilon }$,
\begin{eqnarray}\label{psi}
\psi _ { \epsilon } ( x ) &=& \frac{\tilde{u}\left(\tilde{x}_\epsilon+ r _ { \epsilon } x \right)} {c _ { \epsilon }},\\
\label{phi}
\varphi _{\epsilon}\left( x \right) &=&c_{\epsilon}\left( \tilde{u}\left(\tilde{x}_\epsilon+r_{\epsilon}x \right) -c_{\epsilon} \right).
\end{eqnarray}
Now we study the convergence behavior of $\psi _{\epsilon}$ and\ $\varphi _{\epsilon}.$

\begin{lemma}
Up to a subsequence, there hold
\begin{eqnarray}
\label{1}  \lim_{ \epsilon\rightarrow0} \psi _{\epsilon}(x) &=& 1  \ \ \ \mathrm{in}\ \ \ C_{loc}^{1}( \mathbb{R}^2 ),\\
\label{6}
  \lim_{ \epsilon\rightarrow0}\varphi _{\epsilon}(x)&=&\varphi(x) \ \ \ \mathrm{in}\  \ \  C_{loc}^{1}( \mathbb{R}^2 ),
\end{eqnarray}
where \begin{eqnarray}\label{7}
\varphi \left( x \right) =-\frac{1}{2\pi}\log \left( 1+\frac{\pi}{2} \left| x \right|^2 \right)
\end{eqnarray}
and
\begin{eqnarray}\label{8}
\int_{\mathbb{R}^2_+}e^{4\pi \varphi(x)}dx=1.
\end{eqnarray}

\end{lemma}

\begin{proof}
By (\ref{e-u}) and (\ref{r})-(\ref{phi}), a direct computation shows
\begin{eqnarray}\label{e-psi}
   -\Delta _{\mathbb{R}^2}\psi _{\epsilon}&=&\left( c_{\epsilon}^{-2}\psi _{\epsilon}e^{\alpha _{\epsilon}\left( \psi _{\epsilon}+1 \right) \varphi _{\epsilon}}+r_{\epsilon}^{2}\gamma _{\epsilon}\psi _{\epsilon}-\frac{r_{\epsilon}^{2}\mu _\epsilon}{c_{\epsilon}\lambda _{\epsilon}} \right) e^{2\tilde{f}\left(\tilde{x}_\epsilon+r_{\epsilon}x \right)},\\
   \label{e-phi}
-\Delta_{\mathbb{R}^2}\varphi _{\epsilon}&=&\left( \psi _{\epsilon}e^{\alpha _{\epsilon}\left( \psi _{\epsilon}+1 \right) \varphi _{\epsilon}}+c_{\epsilon}^2r_{\epsilon}^{2}\gamma _{\epsilon}\psi _{\epsilon}-\frac{c_{\epsilon}r_{\epsilon}^{2}\mu _{\epsilon}}{\lambda _{\epsilon}} \right) e^{2\tilde{f}\left( \tilde{x}_\epsilon+r_{\epsilon}x \right)}.
	\end{eqnarray}
Since $\left| \psi _{\epsilon} \right|\leq 1$ and $ \lim_{{\epsilon}\rightarrow 0}-\Delta_{\mathbb{R}^2} \psi _{\epsilon}=0$,
we have by the elliptic estimate to (\ref{e-psi}) that $\lim_{{\epsilon}\rightarrow 0}\psi _{\epsilon}=\psi$ in $ C_{loc}^{1}( \mathbb{R}^2 )$, where $\psi$ is a bounded harmonic function in $\mathbb{R}^2$. Note that $\psi \left( 0 \right) =\lim_{\epsilon \rightarrow 0}\,\psi _{\epsilon}\left( 0 \right) =1$. It follows from the Liouville theorem that $\psi \equiv 1$ in $\mathbb{R}^2$. That is to say (\ref{1}) holds.

Note that $\varphi _{\epsilon}\left( x \right) \leq \varphi _{\epsilon}\left( 0 \right) =0$ for any $ x\in U_\epsilon$. Applying Lemma \ref{L4} and the elliptic estimate to (\ref{e-phi}), we obtain (\ref{6}),
where $\varphi$ satisfies
\begin{eqnarray*}
\left\{ \begin{aligned}
&-\Delta_{\mathbb{R}^2} \varphi =e^{4\pi  \varphi}\,\,\text{in\,\,}\mathbb{R}^2,\\
&\varphi \left( 0 \right) =0=\sup_{\mathbb{R}^2}\,\varphi,\\
&\int_{\mathbb{R}^2}{e^{4\pi  \varphi}dx\leq 2}.
	\end{aligned} \right.
\end{eqnarray*}
By the uniqueness theorem in Chen-Li \cite{CL}, we have (\ref{7}).
Moreover, a simple calculation gives
\begin{eqnarray}\label{2}
\int_{\mathbb{R}^2}{e^{4\pi  \varphi}dx= 2}.
\end{eqnarray}
Denote $U_{\epsilon}^{+}=\left\{x \in \mathbb{R}^{2}: \tilde{x}_\epsilon +r_{\epsilon}x\in\mathbb { B }_r^{+}\right\}$ and $U_{\epsilon}^{-}=\left\{x \in \mathbb{R}^{2}:\tilde{x}_\epsilon +r_{\epsilon}x\in\mathbb { B }_r^{-}\right\} .$ For any fixed $R>0$,
let $\mathbb{B}_{R}^{\prime}=\left\{x\in \mathbb{B}_{R}: \tilde{x}_\epsilon +r_{\epsilon}x\in\mathbb { B }_r^{+}\right\}$ and $\mathbb{B}_{R}^{\prime \prime}=\left\{x \in \mathbb{B}_{R}:\tilde{x}_\epsilon +r_{\epsilon}x\in\mathbb { B }_r^{-}\right\}$, we have
\begin{eqnarray*}
\int_{\mathbb{B}_{R}} e^{4 \pi \varphi} dx&=&\lim _{\epsilon \rightarrow 0} \int_{\mathbb{B}_{R}} \frac{1}{\beta_{\epsilon}} \psi_{\epsilon}^{2} e^{\alpha_{\epsilon}\left(1+\psi_{\epsilon}\right) \varphi_{\epsilon}} dx \\
&=&\lim _{\epsilon \rightarrow 0} \int_{\mathbb{B}_{R r_{\epsilon}}\left(\tilde{x}_\epsilon\right)} \frac{1}{\lambda_{\epsilon}} \tilde{u}_{\epsilon}^{2} e^{\alpha_{\epsilon} \tilde{u}_{\epsilon}^{2}} dx\\
& \leq& \lim _{\epsilon \rightarrow 0} \int_{\mathbb{B}_{R r_{\epsilon}}^{+}\left(\tilde{x}_\epsilon\right)} \frac{1}{\lambda_{\epsilon}} \tilde{u}_{\epsilon}^{2} e^{\alpha_{\epsilon} \tilde{u}_{\epsilon}^{2}} dx+\lim _{\epsilon \rightarrow 0} \int_{\mathbb{B}_{R r_{\epsilon}}^{-}\left(\tilde{x}_\epsilon\right)} \frac{1}{\lambda_{\epsilon}} \tilde{u}_{\epsilon}^{2} e^{\alpha_{\epsilon} \tilde{u}_{\epsilon}^{2}} dx.
\end{eqnarray*}
This inequality together with $\int_{U_\epsilon} \tilde{u}_{\epsilon}^{2} e^{\alpha_{\epsilon} \tilde{u}_{\epsilon}^{2}}dx \leq \lambda_{\epsilon}$ and (\ref{2}) gives
\begin{eqnarray*}
\lim_{R \rightarrow+\infty} \lim_{\epsilon \rightarrow 0} \int_{\mathbb{B}_{R r_{\epsilon}}^{+}\left(\tilde{x}_\epsilon\right)} \frac{1}{\lambda_{\epsilon}} \tilde{u}_{\epsilon}^{2} e^{\alpha_{\epsilon} \tilde{u}_{\epsilon}^{2}} dx&=&1, \\
\lim_{R \rightarrow+\infty} \lim_{\epsilon \rightarrow 0} \int_{\mathbb{B}_{R r_{\epsilon}}^{-}\left(\tilde{x}_\epsilon\right)} \frac{1}{\lambda_{\epsilon}} \tilde{u}_{\epsilon}^{2} e^{\alpha_{\epsilon} \tilde{u}_{\epsilon}^{2}} dx&=&1.
\end{eqnarray*}
That is to say (\ref{8}) holds. Then we have the lemma.
\end{proof}

Next we discuss the convergence behavior of $u_\epsilon$ away from $x_0$.
Denote $u_{\epsilon,\,\beta }=\min  \{\ \beta c_{\epsilon} , u_{\epsilon} \}  \in W^{1,2}\left(\Sigma, g \right) $ for any real number $0<\beta<1$. Following (\cite{Yang2006IJM}, Lemma 3.6), we get
\begin{eqnarray}\label{9}
\lim_{\epsilon \rightarrow 0}\left\| \nabla_g u _ { \epsilon ,\ \beta } \right\| _ { 2 } ^ { 2 }=\beta.
\end{eqnarray}

\begin{lemma}\label{L6}
Letting $\lambda _{\epsilon}$ be defined by (\ref{e-u}), we obtain
\begin{linenomath}
\begin{flalign*}
\begin{split}
&(i)\,\limsup_{ \epsilon \rightarrow 0 }F_{\alpha}^{2\pi-\epsilon}(u_\epsilon)
= {\rm Area}( \Sigma)+\lim_{\epsilon \rightarrow 0}\,\frac{\lambda _{\epsilon}}{c_{\epsilon}^{2}}, \\
&(ii)\,\lim_{\epsilon \rightarrow 0}\frac{\lambda _{\epsilon}}{c_{\epsilon}^{2}}
=\lim_{R\rightarrow +\infty}  \lim_{\epsilon \rightarrow 0}\int_{\phi ^{-1}\left( \mathbb{B}^+_{Rr_{\epsilon}}( \tilde{x}_{\epsilon} )\ri )}{e^{\alpha_\epsilon u_{\epsilon}^{2}}}\,dv_g.
\end{split}&
\end{flalign*}
\end{linenomath}
\end{lemma}

\begin{proof}
Recalling (\ref{e-u}) and (\ref{9}), for any real number $0<\beta<1$, one gets
\begin{eqnarray*}
F_{\alpha}^{2\pi-\epsilon}(u_\epsilon)-{\rm Area}( \Sigma)
&=&
\int_{\left\{ x\in \Sigma :\, u_{\epsilon}\leq \beta c_{\epsilon} \right\}}{( e^{\alpha _{\epsilon}u_{\epsilon}^{2}}-1 )\, \,dv_g}+\int_{\left\{ x\in \Sigma :\, u_{\epsilon}>\beta c_{\epsilon} \right\}}{( e^{\alpha _{\epsilon}u_{\epsilon}^{2}}-1 )\, \,dv_g}
\\
&\leq& \int_{\Sigma}{( e^{\alpha _{\epsilon}u_{\epsilon ,\,\beta}^{2}}-1 )\, \,dv_g}+\frac{u_{\epsilon}^{2}}{\beta ^2c_{\epsilon}^{2}}\int_{\left\{ x\in \Sigma :\,u_{\epsilon}>\beta c_{\epsilon} \right\}}{e^{\alpha _{\epsilon}u_{\epsilon}^{2}}\,dv_g}
\\
&\leq& \int_{\Sigma}e^{\alpha _{\epsilon}u_{\epsilon ,\,\beta}^{2}}\alpha _{\epsilon}u_{\epsilon}^{2}\,dv_g+\frac{\lambda _{\epsilon}}{\beta ^2c_{\epsilon}^{2}}
\\
&\leq& \left( \int_{\Sigma}e^{r\alpha _{\epsilon}u_{\epsilon ,\,\beta}^{2}}\,dv_g \right) ^{1/r}\left( \int_{\Sigma} \alpha _{\epsilon}^su_{\epsilon}^{2s}\,dv_g \right) ^{1/s}+\frac{\lambda _{\epsilon}}{\beta ^2c_{\epsilon}^{2}}.
\end{eqnarray*}
By (\ref{2006Y}) and (\ref{9}), $e^{\alpha _{\epsilon}u_{\epsilon ,\,\beta}^{2}}$ is bounded in $L^r\left( \Sigma, g \right)$ for some $r>1$.
Then letting $\epsilon \rightarrow 0$ first and then $\beta \rightarrow 1$, we obtain
\begin{eqnarray}\label{s111}
\limsup _ { \epsilon \rightarrow 0 } F_{\alpha}^{2\pi-\epsilon}(u_\epsilon)-{\rm Area}( \Sigma)\leq \limsup_{\epsilon \rightarrow 0}\,\frac{\lambda _{\epsilon}}{c_{\epsilon}^{2}}.
\end{eqnarray}
According to $c_{\epsilon}={\max}_{\overline{\Sigma}}u_{\epsilon}$, (\ref{e-u}) and Lemma \ref{L3}, we have
\begin{eqnarray*}
F_{\alpha}^{2\pi-\epsilon}(u_\epsilon)-{\rm Area}( \Sigma) \geq \frac{{\lambda_{\epsilon}}}{c_{\epsilon}^{2}}-
\frac{1}{c_{\epsilon}^{2}}\int_{\Sigma}{u_{\epsilon}^{2}}\,dv_g,
\end{eqnarray*}
that is to say
\begin{eqnarray}\label{s112}
\limsup _ { \epsilon \rightarrow 0 }F_{\alpha}^{2\pi-\epsilon}(u_\epsilon) -{\rm Area}\left( \Sigma \right) \geq \liminf _ {\epsilon\rightarrow0}\frac{\lambda_{\epsilon}}{c_{\epsilon}^{2}}.
\end{eqnarray}
Combining (\ref{s111}) and (\ref{s112}), one gets ($i$).

Applying (\ref{e-u}) and (\ref{r})-(\ref{phi}), we obtain
\begin{eqnarray*}
\int_{\phi ^{-1}\left( \mathbb{B}^+_{Rr_{\epsilon}}( \tilde{x}_{\epsilon} ) \right)}{e^{\alpha _{\epsilon}u_{\epsilon}^{2}}}\,dv_g
&=&\int_{\mathbb{B}^+_{Rr_{\epsilon}}\left( \tilde{x}_{\epsilon} \right)}{e^{\alpha _{\epsilon}u_{\epsilon}^{2}\left( x \right)}e^{2f\left( x \right)}}dx
\\
&=&\int_{\mathbb{B}^+_R\left( 0 \right)}{r_{\epsilon}^{2}e^{\alpha _{\epsilon}c_{\epsilon}^{2}\left( x \right)}e^{\alpha _{\epsilon}\left( \psi _{\epsilon}\left( x \right) +1 \right) \varphi _{\epsilon}\left( x \right)}e^{2f\left( \tilde{x}_{\epsilon}+r_{\epsilon}x \right)}}dx
\\
&=&\frac{\lambda _{\epsilon}}{c_{\epsilon}^{2}}\int_{\mathbb{B}^+_R\left( 0 \right)}\frac{1}{\beta _{\epsilon}}{e^{\alpha _{\epsilon}\left( \psi _{\epsilon}\left( x \right) +1 \right)\, \varphi _{\epsilon}\left( x \right)}e^{2f\left( \tilde{x}_{\epsilon}+r_{\epsilon}x \right)}}dx.
\end{eqnarray*}
Letting $\epsilon \rightarrow 0$ first and then $R \rightarrow +\infty$, we have ($ii$)
by (\ref{6})-(\ref{8}).
\end{proof}

Next we consider the properties of $c_{\epsilon}u_{\epsilon}$.
Using the similar idea of (\cite{Yang2006IJM}, Lemma 3.9), one gets
\begin{eqnarray}\label{10}
 \frac { \beta _ { \epsilon } } { \lambda _ { \epsilon } } c _ { \epsilon } u _ { \epsilon } e ^ { \alpha _ { \epsilon } u _ { \epsilon } ^ { 2 } } \,dv_g \rightharpoonup \delta_{x_0}.
\end{eqnarray}
After a slight modification of (\cite{Yang2007}, Lemma 4.8), we have
\begin{lemma}\label{L5}
  Assume $u \in C ^ { \infty } \left( \overline{\Sigma}\right)$ is a solution of $\Delta_g u = f ( x )$ in $(\Sigma, g)$ and satisfies $ { \| u \| _ { 1 } \leq c _ { 0 } \| f \| _ { 1 } }$. Then for any $1 < q < 2 $, there holds $\| \nabla_g u \| _ { q } \leq C \left( q , c _ { 0 } , \Sigma,\ g \right) \| f \| _ { 1 }.$
\end{lemma}

\begin{lemma}\label{LG}
  For any $1 < q < 2$, $c _ { \epsilon } u _ { \epsilon }$ is bounded in $W ^ { 1 , q } (\Sigma,  g)$.
  Moreover, there holds
	\begin{eqnarray*}
	\left\{ \begin{aligned}
	&c_{\epsilon}u_{\epsilon}\rightharpoonup G\, \, \mathrm{weakly\, \, in\,}\, W^{1,q}\left(\Sigma,  g \right), \,\forall 1<q<2, \\
	&c_{\epsilon}u_{\epsilon}\rightarrow G\, \, \mathrm{strongly\, \, in  \,}\, L^s\left(\Sigma,  g \right), \,\forall 1<s<\frac{2q}{2-q}, \\
	&c_{\epsilon}u_{\epsilon}\rightarrow G\ \mathrm{in\,}\, {C_{loc}^{1}\left( \Sigma\backslash \left\{ x_0 \right\} \right) },
	\end{aligned} \right.
	\end{eqnarray*}
where $G$ is a Green function satisfying \begin{eqnarray}\label{e-G}
	\left\{ \begin{aligned}
&\Delta_g G=\delta _{x_0}+\alpha G-\frac{1}{{\rm Area}\left( \Sigma \right)}\  {\mathrm{in}}\  \Sigma,\\
&\frac{\partial G}{\partial \mathbf{n}}=0\  {\mathrm{on}}\  {\partial \Sigma}\backslash \left\{ x_0 \right\},\\
&\int_{\Sigma}{G\,dv_g}=0.
	\end{aligned} \right.
\end{eqnarray}
\end{lemma}

\begin{proof}
It follows from (\ref{e-u}) that
 	\begin{eqnarray}\label{11}
\Delta_g \left( c_{\epsilon}u_{\epsilon} \right) =\frac{\beta _{\epsilon}}{\lambda _{\epsilon}}c_{\epsilon}u_{\epsilon}e^{\alpha _{\epsilon}u_{\epsilon}^{2}}+\gamma _{\epsilon}c_{\epsilon}u_{\epsilon}-c_{\epsilon}\frac{\mu _{\epsilon}}{\lambda _{\epsilon}}.
	\end{eqnarray}
According to (\ref{e-u}) and (\ref{10}), we have
\begin{eqnarray}\label{12}
\left| \frac{c_{\epsilon}\mu _{\epsilon}}{\lambda _{\epsilon}} \right|=\frac{1}{\text{Area}\left( \Sigma \right)}\int_{\Sigma}{\frac{\beta _{\epsilon}}{\lambda _{\epsilon}}}c_{\epsilon}u_{\epsilon}e^{\alpha _{\epsilon}u_{\epsilon}^{2}}\,dv_g
= \frac{1}{{\rm Area}\left( \Sigma \right)}\left( 1+o_{\epsilon}\left( 1 \right) \right).
\end{eqnarray}
In view of Lemma \ref{L3}, (\ref{e-u}), (\ref{4}), (\ref{10}), (\ref{11}) and (\ref{12}), we have $\Delta_g \left( c _ { \epsilon } u _ { \epsilon } \right)$ is bounded in $L^1(\Sigma,  g)$.
From Lemma \ref{L5}, there holds $c _ { \epsilon } u _ { \epsilon }$ is bounded in $W ^ { 1 , q } (\Sigma,  g)$ for any $1 < q < 2$.
Then $c_{\epsilon}u_{\epsilon}\rightharpoonup G\, \, \mathrm{weakly\, \, in\,}\, W^{1,q}\left(\Sigma,  g \right)$ for any $1<q<2$ and $c_{\epsilon}u_{\epsilon}\rightarrow G\, \, \mathrm{strongly\, \, in  \,}\, L^s\left(\Sigma, g \right)$ for any $1<s<{2q}/{(2-q)}$.

We choose a cut-off function $\rho$ in $C^{\infty}\left( \overline{\Sigma} \right)$, which is equal to $0$ in $\overline{B_\delta(x_0)}$ and equal to $1$ in $\Sigma\backslash B_{2\delta}(x_0)$ such that
$\lim_{\epsilon\rightarrow0}{\lVert \nabla_g \left( \rho u_{\epsilon} \right) \rVert _{2}^{2}}=0.$
Hence there holds
\begin{eqnarray*}
  \int_{\Sigma \backslash B_{2\delta}\left( x_0 \right)}{e^{s\alpha _{\epsilon}u_{\epsilon}^{2}}}dx\leqslant \int_{\Sigma \backslash B_{2\delta}\left( x_0 \right)}{e^{s\alpha _{\epsilon}\lVert \nabla_g \left( \rho u_{\epsilon} \right) \rVert _{2}^{2}\frac{\rho ^2u_{\epsilon}^{2}}{\lVert \nabla_g \left( \rho u_{\epsilon} \right) \rVert _{2}^{2}}}}dx.
\end{eqnarray*}
From the Trudinger-Moser inequality (\ref{2006Y}),  ${e^{\alpha _{\epsilon}u_{\epsilon}^{2}}}$ is bounded in $L^s\left( \Sigma, g \right) $ for some $s>1$.
Applying the elliptic estimate and the compact embedding theorem to (\ref{11}), we obtain $c_{\epsilon}u_{\epsilon}\rightarrow G$ in $ C_{loc}^{1}\left( {\Sigma }\backslash \left\{ x_0 \right\} \right).$
Testing (\ref{11}) by $\phi \in C^1\left( \Sigma \right)$, we obtain (\ref{e-G}).
	\end{proof}

Applying the elliptic estimate, we can decompose $G$ as the form
\begin{eqnarray}\label{G}
  G=-\frac{1}{\pi}\log |x-x_0|+A_{x_0}+\sigma(x),
\end{eqnarray}
where $A_{x_0}$ is a constant only on $x_0$ and $\sigma(x)\in C^{\infty}\left( \overline{\Sigma}\right)$ with $\sigma(x_0)=0$.

\subsection*{\textbf{Step 3.} Upper bound estimate}
To derive an upper bound of $\sup _ { u \in \mathcal { S} } F _\alpha^ { 2 \pi }( u )$, we use the capacity estimate, which was first used by Li \cite{Li-JPDE} in this topic.

	\begin{lemma}\label{L8}
There holds
	\begin{eqnarray*}\label{15}
	\sup _ { u \in \mathcal { S} } F^ { 2 \pi } _ { \alpha } ( u ) \leq {\rm Area}( \Sigma) + \frac{\pi}{2} e ^ { 1 + 2 \pi A _ { x_0} }.
    \end{eqnarray*}
	\end{lemma}

\begin{proof}
We take an isothermal coordinate system $(U,\phi)$ near $x_0$ such that $\phi(x_0) = 0$, $\phi (U)\subset \mathbb { R } ^ { 2  }_+$ and $\phi (U\cap \partial\Sigma)\subset \partial\mathbb { R } ^ { 2  }_+$. In such coordinates, the metric $g$ has the representation $g = e^{2f} \left(dx_1^2 +dx_2^2\right )$ and $f$ is a smooth function with $f(0)  = 0$.
Denote $\tilde{u}_\epsilon=u_\epsilon\circ\phi^{-1}$.
We claim that
\begin{eqnarray}\label{16}
\lim_{\epsilon \rightarrow 0}\ \frac{\lambda _{\epsilon}}{c_{\epsilon}^{2}}\leq \frac{\pi}{2} e ^ { 1 + 2 \pi A _ { x_0} }.
\end{eqnarray}

To confirm this claim, we set
$ a={\sup_ { \partial \mathbb { B }_ { \delta }  \cap \mathbb { R } ^ { 2 }_+ } }\tilde{u}_\epsilon$ and $b=\inf _ { \partial \mathbb { B }_{Rr_\epsilon}\cap \mathbb { R } ^ { 2 }_+ }\tilde{u}_\epsilon$ for sufficiently small $\delta>0$ and some fixed $R>0$.
According to (\ref{6}), (\ref{7}), (\ref{G}) and Lemma \ref{LG}, one gets
\begin{eqnarray*}
 a=\frac{1}{c_{\epsilon}}\left( \frac{1}{\pi}\log \frac{1}{\delta}+A_{x_0}+o_\delta(1)+o_{\epsilon}(1)\right),\ \ \\
b=c_{\epsilon}+\frac{1}{c_{\epsilon}}\left( -\frac{1}{2\pi}\log \left( 1+\frac{\pi}{2} R^{2} \right) +o_{\epsilon}( 1 ) \right),
\end{eqnarray*}
where $o_\delta(1)\rightarrow0$ as $\delta\rightarrow 0$ and $o_{\epsilon}(1)\rightarrow0$ as $\epsilon\rightarrow 0$.
It follows from a direct computation that
\begin{eqnarray}\label{17}
\pi ( a-b ) ^2=\pi c_{\epsilon}^{2}+2\log \delta -2\pi A_{x_0}-\log \left( 1+\frac{\pi}{2}  R^2\ri ) +o_{\delta}( 1 ) +o_{\epsilon}( 1 ).
\end{eqnarray}
Define a function space
\begin{eqnarray*}
W_{a,b}=\left\{ \tilde{u}\in W^{1,2} \left(\mathbb{B}^+_ \delta\setminus \mathbb{B}^+ _{Rr_\epsilon} \right)  :\
\left.\tilde{u}\right|_{ \partial \mathbb { B }_ { \delta }  \cap \mathbb { R } ^ { 2 }_+ }=a,\,\,
\left.\tilde{u}\right|_{ \partial \mathbb { B }_{Rr_\epsilon}\cap \mathbb { R } ^ { 2 }_+ }=b,\,\,
\left.\frac{\partial \tilde{u}}{\partial \mathbf{v}}\right|_{\partial \mathbb{R}^{2}_+ \cap \left(\mathbb{B}_{\delta} \backslash \mathbb{B}_{R r_{\epsilon}}\right)}=0 \right\},
\end{eqnarray*}
where $\mathbf{v}$ denotes the outward unit normal vector on $\partial{ \mathbb { R } ^ { 2 }_+ }$.
Applying the direct method of variation, we obtain
$\inf_{u\in W_{a,b}} \int_{\mathbb{B}^+_ \delta \setminus \mathbb{B}^+ _{Rr_\epsilon}  }{|\nabla_{\mathbb{R}^2} u|^2}dx$ can be attained by some function $m(x)\in W_{a,b}$ with $\Delta_{\mathbb{R}^2} m( x ) =0$.
We can check that
\begin{eqnarray*}
m( x ) =\frac{a\left( \log|x|-\log ( {Rr_\epsilon} ) \right) +b\left( \log \delta -\log|x| \right)}{\log \delta -\log ( {Rr_\epsilon})}
\end{eqnarray*}
and
\begin{eqnarray}\label{19}
\int_{\mathbb{B}^+_ \delta \setminus \mathbb{B}^+ _{Rr_\epsilon}  }{|\nabla_{\mathbb{R}^2} m( x ) |^2}dx=\frac{\pi ( a-b) ^2}{\log \delta -\log ( {Rr_\epsilon} )}.
\end{eqnarray}
Recalling (\ref{e-u}) and (\ref{r}), we have
\begin{eqnarray}\label{20}
\log \delta -\log ( {Rr_\epsilon} ) =\log \delta -\log R-\frac{1}{2}\log \frac{\lambda _{\epsilon}}{\beta _{\epsilon}c_{\epsilon}^{2}}+\frac{1}{2}\alpha _{\epsilon}c_{\epsilon}^{2}.
\end{eqnarray}
Letting $u^*_{\epsilon}=\max \left\{ a,\ \min \left\{b,\ \tilde{u}_{\epsilon} \right\} \right\} \in W_{a,b}$, one gets $| \nabla_{\mathbb{R}^2} u^*_{\epsilon} |\leq | \nabla_{\mathbb{R}^2} \tilde{u}_{\epsilon} |$ in $\mathbb{B}^+_ \delta \setminus \mathbb{B}^+ _{Rr_\epsilon}  $ for sufficiently small $\epsilon$.
According to this and $\|\nabla_g u_{\epsilon}\|_2^2=1$, we obtain
\begin{eqnarray}
 \nonumber\int_{\mathbb{B}^+_ \delta \setminus \mathbb{B}^+ _{Rr_\epsilon}  }{|\nabla _{\mathbb{R}^2}m( x ) |^2}dx
 &\leq& \int_{\mathbb{B}^+_ \delta \setminus \mathbb{B}^+ _{Rr_\epsilon}  }{|\nabla _{\mathbb{R}^2}u_{\epsilon}^{*}( x ) |^2}dx\\
\label{21}&\leq& 1-\int_{\Sigma \backslash \phi ^{-1}\left(\mathbb{B}^+ _ \delta \right)}{| \nabla_g u_{\epsilon} |}^2\,dv_g-\int_{\phi ^{-1}\left( \mathbb{B}^+ _{Rr_\epsilon}  \right)}{| \nabla_g u_{\epsilon} |}^2\,dv_g.
\end{eqnarray}
Now we compute $\int_{\Sigma \backslash \phi ^{-1}\left(\mathbb{B}^+ _ \delta \right)}{| \nabla_g u_{\epsilon} |}^2\,dv_g$ and $\int_{\phi ^{-1}\left( \mathbb{B}^+ _{Rr_\epsilon}  \right)}{| \nabla_g u_{\epsilon} |}^2\,dv_g$.
In view of (\ref{e-G}) and (\ref{G}), we obtain
\begin{eqnarray*}
\int_{\Sigma \backslash \phi ^{-1}\left(\mathbb{B}^+ _ \delta \right)}{|}\nabla_g G|^2\,dv_g=\frac{1}{\pi}\log \frac{1}{\delta} +A_{x_0}+\alpha \lVert G \rVert _2^{2}+o_{\epsilon}( 1 ) +o_{\delta}( 1 ).
\end{eqnarray*}
Hence we have by Lemma \ref{LG}
\begin{eqnarray}\label{G3}
\int_{\Sigma \backslash \phi ^{-1}\left(\mathbb{B}^+ _ \delta \right)}{| \nabla_g u_{\epsilon} |}^2\,dv_g=\frac{1}{c_{\epsilon}^{2}}\left(\frac{1}{\pi}\log \frac{1}{\delta} +A_{x_0}+\alpha \lVert G \rVert _2^{2}+o_{\epsilon}( 1 ) +o_{\delta}( 1 )\right).
\end{eqnarray}
It follows from (\ref{phi}), (\ref{6}) and (\ref{7}) that
\begin{eqnarray}\label{22}
\int_{\phi ^{-1}\left( \mathbb{B}^+ _{Rr_\epsilon}  \right)}{| \nabla_g u_{\epsilon} |}^2\,dv_g=\frac{1}{c_{\epsilon}^{2}}\left( \frac{1}{2\pi}\log \left( 1+\frac{\pi}{2} R^2  \right) -\frac{1}{2\pi}+o_\epsilon\left( 1 \right)+o_{R}\left( 1 \right)\right),
\end{eqnarray}
where $o_R( 1 ) \rightarrow 0$ as $R\rightarrow +\infty $.
Recalling (\ref{17})-(\ref{22}), we obtain
\begin{eqnarray*}
\log\frac{\lambda _{\epsilon}}{c_{\epsilon}^{2}}\le \log \frac{\pi}{2}+1+2\pi A_{x_0} +o(1),
\end{eqnarray*}
where $o( 1 ) \rightarrow 0$ as $\epsilon \rightarrow 0$ first, then $R\rightarrow +\infty $ and $\delta \rightarrow 0$.
Hence (\ref{16}) is followed.
Combining (\ref{3}), (\ref{16}) and Lemma \ref{L6}, we finish the proof of the lemma.
\end{proof}
From Lemma \ref{L8}, the proof of Theorem \ref{T1} ($ii$) follows immediately under the hypothesis of $c_{\epsilon}\rightarrow +\infty$.

\section{Existence of the Extremal Functions}
The content in this section is carried out under the condition $0 \leq \alpha < \lambda_1(\Sigma)$ and $c_{\epsilon}\rightarrow +\infty$.
Set a cut-off function
$\xi \in C _ { 0 } ^ { \infty } \left( B _ { 2 R \epsilon } ( x_0 ) \right)$ with $\xi = 1$ on $\overline{B _ { R \epsilon } ( x_0)}$ and $\| \nabla_g \xi \| _ { L ^ { \infty } } = O (  (  R \epsilon )^{-1} )$. Denote $\tau =G+(\pi)^{-1}\log |x-x_0|-A_{x_0}$, where $G$ is defined as in (\ref{G}). Let $R= -\log \epsilon $, then $R\rightarrow+\infty$ and $R\epsilon\rightarrow0$ as $\epsilon\rightarrow0$.
We construct a blow-up sequence
	\begin{eqnarray}\label{23}
    v_{\epsilon}=\left\{\begin{aligned}
	&\frac{c^2-\frac{1}{2\pi}{\log \left( 1+\frac{\pi}{2} \frac{|x-x_0|^2}{\epsilon ^2} \right)} +b}{\sqrt{c^2+\alpha \lVert G \rVert _2^{2}}},&		\,\,& x\in {B_{R\epsilon}\left( x_0 \right) },\\
	&\frac{G-\xi \tau}{\sqrt{c^2+\alpha \lVert G \rVert _2^{2}}},&		\,\,\ \ &x\in B_{2R\varepsilon}( x_0 )\backslash B_{R\varepsilon}( x_0 ),\\
	&\frac{G}{\sqrt{c^2+\alpha \lVert G \rVert _2^{2}}},&		\,\,&x\in \Sigma \backslash B_{2R\varepsilon}( x_0 ),\\
\end{aligned} \right.
	\end{eqnarray}
where $b$ and $c$ are constants to be determined later.
In order to assure that $v _{\epsilon}\in  C^{\infty}\left( \overline{\Sigma}\right)$, we obtain
\begin{eqnarray}\label{c1}
c^2-\frac{1}{2\pi}\log \left( 1+\frac{\pi}{2} R^2 \right) +b=-\frac{1}{\pi}\log \left( R\epsilon\right) +A_{x_0}.\end{eqnarray}
It follows from $\lVert \nabla_g v _{\epsilon} \rVert _2=1$ that
\begin{eqnarray}\label{c2}
{c}^2=A_{x_0}-\frac{1}{\pi}\log \epsilon+\frac{1}{2\pi}\log\frac{ \pi}{2}-\frac{1}{2\pi}+O\left(\frac{1}{R^2}\right)+O\left( R\epsilon \log ( R\epsilon ) \right) +o_{\epsilon}\left( 1 \right).
\end{eqnarray}
In view of (\ref{c1}) and (\ref{c2}), we have
	\begin{eqnarray}\label{B}
b = \frac { 1 } { 2 \pi } + O\left(\frac{1}{R^2}\right)+ O \left( R \epsilon \log ( R \epsilon ) \right)+o_{\epsilon}\left( 1 \right).
	\end{eqnarray}
A delicate and simple calculation shows
\begin{eqnarray}\label{28}
\lVert v_{\epsilon} \rVert _{2}^{2}=\frac{\lVert G \rVert _{2}^{2}+O\left( R\epsilon \log \left( R\epsilon \right) \right)}{c^2+\alpha \lVert G \rVert _{2}^{2}}\geq \frac{\lVert G \rVert _{2}^{2}+O\left( R\epsilon \log \left( R\epsilon \right) \right)}{c^2}\left( 1-\frac{\alpha \lVert G \rVert _{2}^{2}}{c^2} \right),
\end{eqnarray}
which gives on $\left(B_{R\epsilon}\left( x_0 \right) , g\right)$
\begin{eqnarray*}2\pi v_{\epsilon}^{2}\left( 1+\alpha \lVert v_{\epsilon} \rVert _{2}^{2} \right) \ge 2\pi c^2+4\pi b-\text{2}\log \left( 1+\frac{\pi}{2}\frac{|x-x_0|^2}{\epsilon ^2} \right) -\frac{4\pi \alpha ^2\lVert G \rVert _{2}^{4}}{c^2}+O\left( \frac{\log R}{c^4} \right) .\end{eqnarray*}
Denote $ { v }^* _ { \epsilon } =   \int _ { \Sigma } v _ { \epsilon } \,dv_g/{ {\rm Area}( \Sigma) }$.
It is easy to know that $ { v }^* _ { \epsilon } = O \left( ( R \epsilon ) ^ { 2 } \log \epsilon\right)$ and $v_{\epsilon}-v_{\epsilon}^{*} \in \mathcal{S}$.
On the one hand, by (\ref{c1})-(\ref{28}), there holds
	\begin{eqnarray}\label{27}
\int_{B_{R\epsilon}\left( x_0 \right)}{e}^{2\pi \left( v_{\epsilon}-v_{\epsilon}^{*} \right) ^2\left( 1+\alpha \lVert v_{\epsilon}-v_{\epsilon}^{*} \rVert _{2}^{2} \right)}\,dv_g\ge \frac{\pi}{2}e^{1+2\pi A_{x_0}}-\frac{2\pi ^2\alpha ^2\lVert G \rVert _{2}^{4}}{c^2}e^{1+2\pi A_{x_0}}+O\left( \frac{\log R}{c^4} \right) +O\left( \frac{\log\log \epsilon}{R^2} \right).
	\end{eqnarray}	
On the other hand, from the fact of $e^t\geq t+1$ for any $t\geq0$ and (\ref{23}), one gets
	\begin{eqnarray}
\nonumber\int_{\Sigma \backslash B_{R\epsilon}( x_0)}{e}^{2\pi \left( v_{\epsilon}-v_{\epsilon}^{*} \right) ^2\left( 1+\alpha \lVert v_{\epsilon}-v_{\epsilon}^{*} \rVert _2^{2}\right)}\,dv_g
&\geq& \int_{\Sigma \backslash B_{2R\varepsilon}( x_0 )}{\left( 1+2\pi ( v_{\epsilon}-v_{\epsilon}^{*} ) ^2\right)}\,dv_g\\
   \label{24} &\geq& \text{Area}\left( \Sigma \right) +2\pi \frac{\lVert G \rVert _{2}^{2}}{c^2}+O\left( \frac{\log R}{c^4} \right) +O\left( R^2\epsilon^2 \right) .
	\end{eqnarray}
It follows from (\ref{27}) and (\ref{24}) that
	\begin{eqnarray*}\label{25}
\int_{\Sigma}{e}^{2\pi \left( v_{\epsilon}-v_{\epsilon}^{*} \right) ^2\left( 1+\alpha \lVert v_{\epsilon}-v_{\epsilon}^{*} \rVert _2^{2} \right)}\,dv_g&\geq
&\text{Area}\left( \Sigma \right) +\frac{\pi}{2}e^{1+2\pi A_{x_0}}+\frac{2\pi \lVert G \rVert _{2}^{2}}{c^2}\left( 1-\pi \alpha ^2\lVert G \rVert _{2}^{2}e^{1+2\pi A_{x_0}} \right) \\
&\ & +O\left( \frac{\log\log \epsilon}{R^2} \right) +O\left( \frac{\log R}{c^4} \right) +O\left( R^2\epsilon^2 \right).	
    \end{eqnarray*}
    According to $R= -\log \epsilon $ and (\ref{c1}), we obtain
	\begin{eqnarray}\label{26}
 F^ { 2 \pi } _ { \alpha } ( v _ { \epsilon } - v^*_ { \epsilon } ) > {\rm Area}( \Sigma) + \frac{\pi}{2} e ^ { 1 + 2 \pi A _ { x_0} }.
    \end{eqnarray}
for sufficiently small $\alpha$ and $\epsilon$.
The contradiction between (\ref{3}) and (\ref{26}) indicates that $c_{\epsilon}$ must be bounded when $\alpha$ is sufficiently small.
When $|c_{\epsilon}|\leq C$, using Lebesgue's dominated convergence, we have
	\begin{eqnarray*}
F^{2\pi}_\alpha(u_0)=\sup_{u\in \mathcal{S}}F^{2\pi}_\alpha(u).
	\end{eqnarray*}
Moreover, it is easy to see $u_0 \in C^\infty\left(\overline{\Sigma}\right)\cap \mathcal {S}$ from Lemma \ref{L1} and (\ref{s2}).
Therefore, we obtain Theorem \ref{T1} ($iii$).

\nolinenumbers

\end{document}